\documentclass[11pt]{amsart}
\usepackage{amssymb}
\usepackage[T1]{fontenc}

\usepackage{anysize}

\marginsize{2.5cm}{2.5cm}{2.5cm}{2.5cm}
\theoremstyle{plain}
\newtheorem{theorem}{Theorem}
\newtheorem{proposition}{Proposition}
\newtheorem{lemma}{Lemma}
\newtheorem{corollary}{Corollary}

\theoremstyle{remark}
\newtheorem{remark}{Remark}

\title{There are infinitely many rational Diophantine sextuples}

\author{Andrej Dujella}
\address{Department of Mathematics, University of Zagreb, Bijeni\v{c}ka cesta 30, 10000 Zagreb, Croatia}
\email{duje@math.hr}
\author{Matija Kazalicki}
\address{Department of Mathematics, University of Zagreb, Bijeni\v{c}ka cesta 30, 10000 Zagreb, Croatia}
\email{matija.kazalicki@math.hr}
\author{Miljen Miki\'c}
\address{Kumi\v{c}i\'ceva 20, 51000 Rijeka, Croatia}
\email{miljen.mikic@gmail.com}
\author{M\'arton Szikszai}
\address{Institute of Mathematics\\University of Debrecen\\Debrecen, P.O. Box 12.\\H-4010, Hungary}
\email{szikszai.marton@science.unideb.hu}

\keywords{Diophantine sextuples, elliptic curve}

\subjclass[2010]{11D09, 11G05, 11Y50}

\begin{document}

%Abstract
\begin{abstract}
A rational Diophantine $m$-tuple is a set of $m$ nonzero rationals such that the product of any two of them increased by $1$ is a perfect square. The first rational Diophantine quadruple was found by Diophantus, while Euler proved that there are infinitely many rational Diophantine quintuples. In 1999, Gibbs found the first example of a rational Diophantine sextuple. In this paper, we prove that there exist infinitely many rational Diophantine sextuples.
\end{abstract}

%Title
\maketitle

%Introduction
\section{Introduction}

\noindent A set of $m$ nonzero rationals $\{a_1,a_2,\ldots,a_m\}$ is called {\em a rational Diophantine $m$-tuple} if $a_ia_j+1$ is a perfect square for all $1\leq i<j\leq m$. An open question is how large a rational Diophantine tuple can be. In the case of integer Diophantine tuples, the corresponding question has been almost completely answered. That is to say, it is well-known and easy to prove that there exist infinitely many integer Diophantine quadruples (e.g. $\{k-1,k+1,4k,16k^3-4k\}$ for $k\geq 2$), while it was proved in \cite{D-crelle} that an integer Diophantine sextuple does not exist and that there are only finitely many such quintuples (for additional information and references see e.g. \cite{D-web}, \cite[Section D29]{Guy}). However, concerning rational Diophantine tuples, no absolute upper bound for the size of such sets is known. The first example of a rational Diophantine quadruple was the set
$$
\left\{\frac{1}{16},\, \frac{33}{16},\, \frac{17}{4},\, \frac{105}{16}\right\}
$$
found by Diophantus (see \cite{Dio}). Euler found infinitely many rational Diophantine quintuples (see \cite{Hea}),  e.g. he was able to extend the integer Diophantine quadruple
$$
\{1,3,8,120\}
$$
found by Fermat, to the rational quintuple
$$
\left\{ 1, 3, 8, 120, \frac{777480}{8288641} \right\}.
$$
Let us note that Baker and Davenport \cite{B-D} proved that Fermat's set cannot be extended to an integer Diophantine quintuple, while Dujella and Peth\H{o} \cite{Du-Pe} showed that there is no integer Diophantine quintuple which contains the pair $\{1,3\}$.
For results on the existence of infinitely many rational $D(q)$-quintuples,
 i.e. sets in which $xy+q$ is always a square, for $q\neq 1$ see \cite{D-F}.

Since 1999, several examples of rational Diophantine sextuples were found by Gibbs \cite{Gibbs1,Gibbs2} and Dujella \cite{D-sex}. The first example, found by Gibbs, was
$$
\left\{ \frac{11}{192}, \frac{35}{192}, \frac{155}{27}, \frac{512}{27}, \frac{1235}{48}, \frac{180873}{16} \right\}.
$$
No example of a rational Diophantine septuple is known. On the other hand, assuming the Lang conjecture on varieties of general type the number of elements of a rational Diophantine tuple can be bound. Indeed, if $\{a_1,a_2,...,a_m\}$ is a rational Diophantine $m$-tuple with $m\ge 5$, then we may consider the hyperelliptic curve
$$
y^2=(a_1x+1)(a_2x+1)(a_3x+1)(a_4x+1)(a_5x+1)
$$
of genus $g=2$. Provided that the above mentioned conjecture holds, Caporaso, Harris and Mazur \cite{C-H-M} proved that for $g\ge 2$ the number  $B(g,\mathbb{K})=\max_{C} |C(\mathbb{K})|$ is finite, where $C$ runs over all curves of genus $g$ over a number field $\mathbb{K}$. Therefore, we get that, under the Lang conjecture, $m \leq 5+B(2,\mathbb{Q})$ (and also  $m\leq 4+B(4,\mathbb{Q})$, see \cite{P-H-Z}).

In this paper, we will use properties of the elliptic curve induced by a rational Diophantine triple to prove that infinitely many rational Diophantine sextuples exist. The following theorem is our main result.
\begin{theorem}
\label{tm1}
There are infinitely many rational Diophantine sextuples. Moreover, there are infinitely many rational Diophantine sextuples with positive elements, and also with any combination of signs.
\end{theorem}
\noindent Since by multiplying all elements of a Diophantine $m$-tuple by $-1$ we obtain a Diophantine $m$-tuple again, it suffices to consider sextuples with zero, one, two and three negative elements.

In brief, this is the way we construct Diophantine sextuples. For a rational number $t\notin \{-1,0,1\}$ let
\begin{equation}\label{eq:E}
E:\quad y^2=x^3+3(t^2-3t+1)(t^2+3t+1)x^2+3(t^2+1)^4x+(t^2+1)^6
\end{equation}
be an elliptic curve with a rational point $R=[0,(t^2+1)^3]$ of infinite order. Rational points on this curve parametrize a certain class of elliptic curves over rational numbers with a square discriminant and prescribed rational point of order $3$. More precisely, to any multiple $[m]R=(x,y)$ of the point $R$ (with $m>1$), we associate an elliptic curve
$$
E':\quad Y^2=X^3+ \sigma_2 X^2 + \sigma_1\sigma_3 X + \sigma_3^2,
$$
where $\sigma_1, \sigma_2$ and $\sigma_3$ are certain explicit rational functions in $x$ and $t$ (e.g. $\sigma_3=\frac{t^2-1}{2t}$). This correspondence is explained in detail in Sections \ref{sec2} and \ref{sec3}.

In Sections \ref{sec4} and \ref{sec5}, we study the primes (and their type) of bad reduction of $E'$ and we show that the curve has rational $2$-torsion for every integer $t$ for which $t^2+1$ is squarefree. Using an effective version of Hilbert's Irreducibility Theorem, we extend this result to any rational $t\notin \{-1,0,1\}$.

Now we can rewrite $E'$ as
$$
Y^2=(X+ab)(X+ac)(X+bc)
$$
for $a,b,c\in \mathbb{Q}$, where $\sigma_1=a+b+c$, $\sigma_2=ab+ac+bc$ and $\sigma_3=abc$. One has that $P'=[0,abc]$ is a point of infinite order, while $S'=[1,\sqrt{(ab+1)(ac+1)(bc+1)}]$ is (prescribed by construction) a rational point of order $3$. Finally, any odd multiple $[2n+1]P'$ of point $P'$ gives a rational Diophantine sextuple
$$
\left\{a,b,c,\frac{x([2n+1]P')}{abc},\frac{x([2n+1]P'+S')}{abc}, \frac{x([2n+1]P'-S')}{abc} \right\},
$$
as explained in Section \ref{sec2}. The elements of the sextuple obtained with the construction, which is briefly explained here, are in fact rational functions of a rational parameter $t$ (for explicit formulas see the end of Section \ref{sec3}).

%Induced elliptic curves
\section{Induced elliptic curves}
\label{sec2}

\noindent Let $\{a,b,c\}$ be a rational Diophantine triple. In order to extend this triple to a quadruple, we have to solve the system \begin{equation}
\label{2e}
ax+1=\Box,\qquad bx+1=\Box,\qquad cx+1=\Box.
\end{equation}
It is natural to assign the elliptic curve
\begin{equation}
\label{3e}
\mathcal{E}: \qquad y^2=(ax+1)(bx+1)(cx+1)
\end{equation}
to the system (\ref{2e}). We say the $\mathcal{E}$ is induced by the triple $\{a,b,c\}$. There are three rational points on the $\mathcal{E}$ of order $2$, namely
$$
A=\left[-\frac{1}{a},0\right],\quad B=\left[-\frac{1}{b},0\right],\quad C=\left[-\frac{1}{c},0\right]
$$
and also other obvious rational points
$$
P=[0,1], \quad S=[1/abc, \sqrt{(ab+1)(ac+1)(bc+1)}/abc ].
$$
The $x$-coordinate of a point $T\in \mathcal{E}(\mathbb{Q})$ satisfies (\ref{2e}) if and only if $T-P\in 2\mathcal{E}(\mathbb{Q})$ (see \cite{D-rim}). It can be verified that $S\in 2\mathcal{E}(\mathbb{Q})$. Indeed, if $ab+1=r^2$, $ac+1=s^2$, $bc+1=t^2$, then $S=[2]R$, where
$$
R=\left[ \frac{rs+rt+st+1}{abc}, \frac{(r+s)(r+t)(s+t)}{abc} \right].
$$
This implies that if $x(T)$ satisfies system (\ref{2e}), then also the numbers $x(T\pm S)$ satisfy the system. It was proved in \cite[Proposition 2]{D-acta2} (with slightly different notation and by direct manipulation with algebraic expressions) that $x(T)x(T\pm S)+1$ is always a perfect square. The following statement extends this result and provides the basis for the proof of Theorem \ref{tm1}.

\begin{proposition}
\label{matija}
Let $Q$, $T$ and $[0,\alpha]$ be three rational points on an elliptic curve $\mathcal{E}$ over $\mathbb{Q}$
given by the equation $y^2=f(x)$, where $f$ is a monic polynomial of degree $3$.
Assume that $\mathcal{O} \not\in \{Q,T,Q+T\}$.
Then
$$ x(Q)x(T)x(Q+T)+ \alpha^2 $$
is a perfect square.
\end{proposition}

\begin{proof}
Consider the curve
$$
y^2 = f(x)-(x-x(Q))(x-x(T))(x-x(Q+T)).
$$
It is a conic which contains three collinear points: $Q$, $T$, $-(Q+T)$ (if the points are distinct) or has a tangent line that intersects the conic in another point (if two of the points are equal). Thus, it is the union of two rational lines, e.g. we have
\begin{equation}
\label{ma1}
y^2 = (\beta x + \gamma)^2.
\end{equation}
Inserting $x=0$ in (\ref{ma1}), we get
$$
x(Q)x(T)x(Q+T)+ \alpha^2 = \gamma^2.
$$
\end{proof}

The coordinate transformation $x\mapsto x/abc$, $y\mapsto y/abc$, applied to the curve $\mathcal{E}$
leads to the elliptic curve
$$ E': \qquad y^2=(x+ab)(x+ac)(x+bc) $$
with a monic cubic polynomial on the right hand side. The points $P$ and $S$ become
$P'=[0,abc]$ and $S'=[1, rst]$, respectively.

If we apply Proposition \ref{matija} with $Q=\pm S'$, since the first coordinate of $S'$ is $1$, we get an elegant proof of the fact that $x(T)x(T\pm S)+1$ is a perfect square (after dividing $x(T')x(T'\pm S') + a^2b^2c^2=\Box$ by $a^2b^2c^2$). Now we have a general construction which produces two rational Diophantine quintuples with four joint elements. So, the union of these two quintuples,
$$
\{a,b,c,x(T-S),x(T),x(T+S)\},
$$
is ``almost'' a rational Diophantine sextuple. Assuming that $T,T\pm S \not\in \{\mathcal{O}, \pm P\}$, the only missing condition is $x(T-S)x(T+ S)+1=\Box$. To construct examples satisfying this last condition, we will use Proposition \ref{matija} with $Q=[2]S'$. To get the desired conclusion, we need the condition $x([2]S')=1$ to be satisfied. This leads to $[2]S'=-S'$, i.e. $[3]S'=\mathcal{O}$. The next lemma characterizes triples $\{a,b,c\}$ satisfying this condition.

\begin{lemma}
\label{miljen}
For the point $S'=[1,rst]$ on $E'$ it holds $[3]S'=\mathcal{O}$ if and only if
\begin{eqnarray}
\label{mi1}
& -a^4b^2c^2+2a^3b^3c^2+2a^3b^2c^3-a^2b^4c^2+2a^2b^3c^3-a^2b^2c^4+12a^2b^2c^2 \nonumber \\
& +6a^2bc+6ab^2c+6abc^2+4ab+4ac+4bc+3=0.
\end{eqnarray}
\end{lemma}

\begin{proof}
The statement of lemma follows directly from the condition $x(S')=x([2]S')$.
\end{proof}

%Construction
\section{Constructions of sextuples and proof of Theorem \ref{tm1}} \label{sec3}

\noindent The polynomial in $a,b,c$ on the left hand side of (\ref{mi1}) is symmetric. Thus, by taking $\sigma_1=a+b+c$, $\sigma_2=ab+ac+bc$, $\sigma_3=abc$, we get from (\ref{mi1}) that
\begin{equation}
\label{mi2}
\sigma_2=(\sigma_1^2 \sigma_3^2-12\sigma_3^2-6\sigma_1 \sigma_3-3)/(4+4\sigma_3^2).
\end{equation}
Inserting (\ref{mi2}) in $(ab+1)(ac+1)(bc+1)=(rst)^2$, we get $(2\sigma_3^2+\sigma_1\sigma_3-1)^2/(4+4\sigma_3^2)=(rst)^2$, i.e. $1+\sigma_3^2=\Box$.

The polynomial $X^3-\sigma_1X^2+\sigma_2X-\sigma_3$ should have rational roots, so its discriminant has to be a perfect square.
Inserting (\ref{mi2}) in the expression for the discriminant, we get
\begin{equation}
\label{mi3}
(\sigma_1^3 \sigma_3-9 \sigma_1^2-27\sigma_1 \sigma_3-54\sigma_3^2-27)(1+\sigma_3^2)(\sigma_1 \sigma_3+2\sigma_3^2-1) = \Box.
\end{equation}
For a fixed $\sigma_3$, we may consider (\ref{mi3}) as a quartic in $\sigma_1$. Since $1+\sigma_3^2$ has to be a perfect square, from (\ref{mi3}) we get a quartic with a rational point (point at infinity), which therefore can be transformed into an elliptic curve.

Let us take $\sigma_3=\frac{t^2-1}{2t}$. Then we get the quartic over
$\mathbb{Q}(t)$ which is birationally equivalent to the following elliptic curve over $\mathbb{Q}(t)$
\begin{eqnarray}
\label{ec:1}
y^2=x^3+(3t^4-21t^2+3)x^2+(3t^8+12t^6+18t^4+12t^2+3)x+(t^2+1)^6.
\end{eqnarray}
This elliptic curve has positive rank, since the point $R=[0,(t^2+1)^3]$ is of infinite order. By taking multiples of the point $R$, transforming these coordinates back to the quartic and computing corresponding triples $\{a,b,c\}$, we may expect to get infinitely many
parametric families of rational triples for which the corresponding point $S'$ on $E'$ satisfies $[3]S'= \mathcal{O}$ (we will provide detailed proof of this claim in Sections \ref{sec4} and \ref{sec5}). Since the condition $1+\sigma_3^2=\Box$ implies $rst \in \mathbb{Q}$, and $S'=-[2]S' \in 2E'(\mathbb{Q})$, an explicit $2$-descent on $E'$ implies that $ab+1$, $ac+1$, $bc+1$ are all perfect squares, thus the triple $\{a,b,c\}$ obtained with this construction is indeed a Diophantine triple. In particular, if we take the point $[2]R$, we get the following family of rational Diophantine triples
\begin{eqnarray*}
a &\!\!=\!\!& \frac{18t(t-1)(t+1)}{(t^2-6t+1)(t^2+6t+1)}, \\
b &\!\!=\!\!& \frac{(t-1)(t^2+6t+1)^2}{6t(t+1)(t^2-6t+1)}, \\
c &\!\!=\!\!& \frac{(t+1)(t^2-6t+1)^2}{6t(t-1)(t^2+6t+1)}.
\end{eqnarray*}
Consider now the elliptic curve over $\mathbb{Q}(t)$ induced by the triple $\{a,b,c\}$. It has positive rank since the point $P=[0,1]$ is of infinite order. Thus, the above described construction produces infinitely many rational Diophantine sextuples containing the triple $\{a,b,c\}$. One such sextuple $\{a,b,c,d,e,f\}$ is obtained by taking $x$-coordinates of points $[3]P$, $[3]P+S$, $[3]P-S$. We get $d=d_1/d_2$, $e=e_1/e_2$, $f=f_1/f_2$, where

{\small
\begin{eqnarray*}
d_1&\!\!=\!\!& 6(t+1)(t-1)(t^2+6t+1)(t^2-6t+1)(8t^6+27t^5+24t^4-54t^3+24t^2+27t+8) \\
& & \mbox{}\times (8t^6-27t^5+24t^4+54t^3+24t^2-
27t+8)(t^8+22t^6-174t^4+22t^2+1), \\
d_2&\!\!=\!\!&t(37t^{12}-885t^{10}+9735t^8-13678t^6+9735t^4-885t^2+37)^2, \\
e_1&\!\!=\!\!&
-2t(4t^6-111t^4+18t^2+25)(3t^7+14t^6-42t^5+30t^4+51t^3+18t^2-12t+2) \\
& & \mbox{}\times (3t^7-14t^6-42t^5-
30t^4+51t^3-18t^2-12t-2)(t^2+3t-2)(t^2-3t-2)\\
& & \mbox{}\times (2t^2+3t-1)(2t^2-3t-1)(t^2+7)(7t^2+1), \\
e_2&\!\!=\!\!&
3(t+1)(t^2-6t+1)(t-1)(t^2+6t+1) \\
& & \mbox{}\times (16t^{14}+141t^{12}-1500t^{10}+7586t^8-2724t^6+165t^4+424t^2-12)^2, \\
f_1&\!\!=\!\!&
2t(25t^6+18t^4-111t^2+4)(2t^7-12t^6+18t^5+51t^4+30t^3-42t^2+14t+3)\\
& & \mbox{}\times (2t^7+12t^6+18t^5-
51t^4+30t^3+42t^2+14t-3)(2t^2+3t-1)(2t^2-3t-1)\\
& & \mbox{}\times (t^2-3t-2)(t^2+3t-2)(t^2+7)(7t^2+1), \\
f_2&\!\!=\!\!&
3(t+1)(t^2-6t+1)(t-1)(t^2+6t+1)\\
& & \mbox{}\times (12t^{14}-424t^{12}-165t^{10}+2724t^8-7586t^6+1500t^4-141t^2-16)^2.
\end{eqnarray*} }%%

These formulas produce infinitely many rational Diophantine sextuples
$\{a,b,c,e,d,f\}$. Moreover, by choosing the rational parameter $t$ from the appropriate interval, we get infinitely many sextuples for each combination of signs. Indeed, for $5.83 < t < 6.86$ all elements are positive, for $t > 6.87$ there is exactly one negative element,
for $1 < t < 1.32$ there are exactly two negative elements, while for $1.33 < t < 2.46$ there are exactly three negative elements in the sextuple. As a specific example, let us take $t=6$, for which we get a sextuple
with all positive elements:
$$
\left \{ \frac{3780}{73}, \frac{26645}{252}, \frac{7}{13140},
\frac{791361752602550684660}{1827893092234556692801}, \right.
$$
$$
\left. \frac{95104852709815809228981184}{351041911654651335633266955},  \frac{3210891270762333567521084544}{21712719223923581005355} \right\}.
$$

%Reduction
\section{More about the construction}
\label{sec4}

\noindent The construction of parametric families of rational Diophantine sextuples in Section \ref{sec3} relies on the fact that the cubic polynomial corresponding to the point $[2]R$ has rational roots. We will show that the same is true for all multiples of $R$. Since the corresponding cubic polynomial has square discriminant (hence it either splits completely over $\mathbb{Q}$ or generates $\mathbb{Z}/3\mathbb{Z}$ Galois extension), it suffices to show that it has at least one rational root.

We will show that $E'$ has a semistable reduction for all odd bad primes, which implies that $E'$ has full rational $2$-torsion (since otherwise the primes dividing the discriminant of $\mathbb{Z}/3\mathbb{Z}$ extension $\mathbb{Q}(E'[2])/\mathbb{Q}$ would be the bad primes of additive reduction).

Let $\sigma_3 \in \mathbb{Q}^\times$ be such that $\sigma_3^2+1$ is a perfect square. Then there is $t\in \mathbb{Q}^\times$ such that $\sigma_3=\frac{t^2-1}{2t}$. We assume throughout the paper that $t \notin  \{-1, 0, 1\}$. Let $P=[x,y]$ with $x \ne 0$ be a rational point on the elliptic curve \eqref{ec:1}. Assume that $x \ne 0$. The corresponding $\sigma_1$ is given by formula
\begin{equation}
\label{eq:0}
\sigma_1 = \frac{-t^4+4t^2-1-x^{-1}(t^2+1)^4}{(t^2-1)t}.
\end{equation}

Consider the curve (here $X$ and $Y$ are variables, not coordinates of the point $P$ on \eqref{ec:1}) $Y^2=X^3+ \sigma_2 X^2 +\sigma_1\sigma_3 X + \sigma_3^2.$ The coordinate transformation $(X,Y) \mapsto (X+1,Y)$ yields
$$
 \qquad Y^2 = X^3+(3+\sigma_2)X^2+(3+2\sigma_2+\sigma_1 \sigma_3) X + (1+\sigma_2+\sigma_1\sigma_3+\sigma_3^2).
$$
By using \eqref{eq:0}, a further change of variables $(X,Y) \mapsto ((\frac{t^2+1}{t})^2X, (\frac{t^2+1}{t})^3 Y)$ yields
\begin{equation}\label{eq:E''}
E'': \qquad Y^2 = X^3+\frac{((t^2+1)^2x^{-1}+1)^2}{4}X^2+\frac{t^2((t^2+1)^2x^{-2}+x^{-1})}{2} X + \frac{t^4x^{-2}}{4}.
\end{equation}
The discriminant $\Delta$ and $c_4$ invariant of $E''$ are equal to
\begin{eqnarray*}
\Delta &=&\displaystyle \frac{t^6y^2}{x^6}, \\
c_4 &=& \displaystyle \frac{((t^2+1)^2x^{-1}+1)(y^2+3x^2t^2)}{x^3}.
\end{eqnarray*}
We will frequently use the fact that if $v_p(\Delta)>0,\;v_p(c_4)=0$ and $E''$ is $p$-integral (i.e. the coefficients of the defining equation are in $\mathbb{Z}_p$), then $E''$ has multiplicative reduction at $p$.
We denote by $v_p$ the standard $p$-adic valuation.

\begin{proposition}
\label{prop:cases}
Assume that $t \in \mathbb{Z}$. Let $p$ be an odd prime.
\begin{itemize}
	\item [a)] If $p|t$ and $x \not \equiv -1 \pmod{p}$, then $E''$ has multiplicative reduction at $p$.
	\item [b)] If $p|t^2+1$ and  $v_p(x) \le 0$, then $E''$ has good or multiplicative reduction at $p$.
	\item [c)] If $p||t^2+1$ and $v_p(x) \ge 4$, $E''$ then has multiplicative reduction at $p$.
\end{itemize}
\end{proposition}

\begin{proof}
First assume that $p|t$.

\noindent\textbf{Case 1.}  $v_p(x)<0$\\
Simple calculation shows that $c_4 \equiv \frac{y^2}{x^3}\equiv 1 \pmod{p}$. Since the curve $E''$ is $p$-integral with $v_p(\Delta)>0$, $E''$ has multiplicative reduction at $p$.

\noindent\textbf{Case 2.} $v_p(x)>0$\\
The coordinate transformation $(X,Y)\mapsto (X/x^2, Y/x^3)$ yields $p$-integral curve with $c_4=((t^2+1)^2+x)(y^2+3x^2t^2)\equiv y^2 \equiv 1 \pmod{p}$. Hence $E''$ has multiplicative reduction at $p$.

\noindent\textbf{Case 3.} $v_p(x)=0$ and $x \not\equiv -1 \pmod{p}$\\
Since $y^2 \equiv (x+1)^3 \pmod{p}$, we have that $c_4 \equiv (x^{-1}+1)y^2/x^3 \not \equiv 0 \pmod{p}$. Curve $E''$ is $p$-integral with $v_p(\Delta)>0$, hence $E''$ has multiplicative reduction at $p$.

Now, assume that $p|t^2+1$.

\noindent\textbf{Case 1.} $v_p(x)<0$\\
In this case $c_4\equiv \frac{y^2}{x^3}+ 3 \frac{t^2}{x} \pmod{p}$. Since
$y^2/x^3 \equiv 1 \pmod{p}$, it follows $c_4 \equiv 1 \pmod{p}$, hence $E''$ has multiplicative reduction at $p$ (since $E''$ is $p$-integral and $v_p(\Delta)>0$).

\noindent\textbf{Case 2.} $v_p(x)\ge 4$\\
The coordinate transformation $(X,Y)\mapsto (X/x^2, Y/x^3)$ yields $p$-integral model with $v_p(\Delta)>0$ and $c_4 = ((t^2+1)^2+x)(y^2+3x^2t^2)$.

If $p||t^2+1$ and $v_p(x)\ge 4$, then $v_p(y) = 3$, hence $v_p(c_4)=8$ and $v_p(\Delta)\ge 6 + 6 v_p(x)$. Therefore, the change of coordinates $(X,Y) \mapsto (p^4X, p^6Y)$ yields a $p$-integral model with $v_p(c_4)=0$, and $v_p(\Delta)>0$, hence $E''$ has multiplicative reduction at $p$.

\noindent\textbf{Case 3.} $v_p(x)=0$\\
In this case $E''$ is $p$-integral and has bad reduction at $p$ if and only if $v_p(y)>0$. In this case (i.e. if $x \equiv -27 \pmod{p}$), $v_p(c_4)=0$ since $v_p(3x^2t^2)= 0$, hence $E''$ has multiplicative reduction at $p$.
\end{proof}

\begin{proposition}
\label{prop:additive}
Assume that $t \in \mathbb{Z}$ and $v_3(y) \le 0$.
If $E''$ has additive reduction at prime $p$, then $p=2$ or $p|(t^2+1)t$.
\end{proposition}

\begin{proof}
Assume that $p\ne 2$ and $p \nmid (t^2+1)t$.

If $v_p(x)>0$, then the change of variables $(X,Y)\mapsto (X/x^2, Y/x^3)$ yields a $p$-integral curve with $v_p(\Delta)=v_p(t^6y^2x^6)>0$. We have that $y^2\equiv (t^2+1)^6 \pmod{p}$, hence $c_4 = ((t^2+1)^2+x)(y^2+3x^2t^2)\equiv (t^2+1)^2(t^2+1)^6 \pmod{p}\not \equiv 0 \pmod{p}$.

If $v_p(y)>0$ (and $v_p(x)\le 0$) then $v_p(x)\ge 0$ (and $p\ne 3$), so the previous model is $p$-integral, and $v_p(\Delta)>0$. If $v_p(c_4)>0$ then $x \equiv -(t^2+1)^2 \pmod{p}$, hence
\begin{align*}
0\equiv y^2&\equiv-(t^2+1)^6+3(t^2-3t+1)(t^2+3t+1)(t^2+1)^4-3(t^2+1)^6+(t^2+1)^6\\
&\equiv (t^2+1)^4(-27t^2) \pmod{p}
\end{align*}
which is a contradiction.
If $v_p(x)<0$, then $E''$ is $p$-integral and $v_p(\Delta)>0$. Since $y^2/x^3 \equiv 1 \pmod{p}$, we have that $c_4 = \displaystyle \frac{((t^2+1)^2x^{-1}+1)(y^2+3x^2t^2)}{x^3}\equiv 1 \pmod{p}$.

If $v_p(y)<0$ then $v_p(x)<0$, so we already proved a multiplicative reduction in this case.

If $v_p(x)=0$ and $v_p(y)=0$, then $E''$ is $p$-integral, and $v_p(\Delta)=0,$ so $E''$ has a good reduction at $p$.
\end{proof}

\begin{proposition}
\label{prop:inertia}
If $E''$ has no rational $2$-torsion, then $E''$ has additive reduction for every odd prime $p|\, disc(\mathbb{Q}(E''[2]))$.
\end{proposition}

\begin{proof}
Since the discriminant of $E''$ is a perfect square, if $E''$ has no rational $2$-torsion (i.e. if $a$, $b$ and $c$ are not rational), then the field generated by the coordinates of $E''[2]$ is $\mathbb{Z}/3\mathbb{Z}$ extension of $\mathbb{Q}$. Define $V_2(E'')=T_2(E'')\otimes \mathbb{Q}_2$, where $T_2(E'')$ is the 2-adic Tate module of $E''$ (we can think of $E''$ as defined over $\mathbb{Q}_p$). Let $I$ be the absolute inertia group of $\mathbb{Q}_p$, and $V_2(E'')^I$ the subspace of $V_2(E'')$ fixed by $I$.
From the definition and the basic properties of the conductor of elliptic curve (see \cite{Sil2}, Chapter 4, \S 10), we know that if $\dim V_2(E'')^I=0$, then the reduction at $p$ is additive (since $p\ne 2$). If $p$ is ramified in $\mathbb{Q}(E''[2])$, then $I$ acts non-trivially on $E''[2]$. Since the ($\mathbb{Z}/3\mathbb{Z}$) action cannot have fixed vectors, we conclude that the reduction at $p$ is additive.
\end{proof}

\begin{corollary}
\label{cor:conclusion}
Let $P=(x,y)$ and $t \in \mathbb{Z}$ be such that $E''$ has good or multiplicative reduction for all $p|t(t^2+1)$ and $v_3(y)\le 0$. Then $E''$ has full rational $2$-torsion.
\end{corollary}

\begin{proof}
Assume that $E''$ does not have rational $2$-torsion. Then Proposition \ref{prop:additive} and Proposition \ref{prop:inertia} imply that $\mathbb{Q}(E''[2])/\mathbb{Q}$ is a cubic $\mathbb{Z}/3\mathbb{Z}$ extension unramified outside $2$. Such extension does not exist (e.g. see John Jones' database of number fields \cite{J-R}).
\end{proof}

%R multiples
\section{Multiples of $R$}
\label{sec5}
\noindent Let $p$ be an odd prime and $t \ne 1$ a positive integer such that $p|| t^2+1$ (in particular, $p>3$). Consider elliptic curve \eqref{ec:1}
$$
E:\qquad y^2=x^3+3(t^2-3t+1)(t^2+3t+1)x^2+3(t^2+1)^4x+(t^2+1)^6,
$$
and point $R=[0,(t^2+1)^3]$ on it. In this section we will prove that the elliptic curve $E''$ that corresponds to the multiples of $R$ has full rational $2$-torsion (i.e. $a,b$ and $c$ are rational). We begin by describing $p$-adic valuation of the coordinates of the multiples $[n]R$.

Denote by $P \mapsto \tilde{P}$, the reduction mod $p$ map, $E(\mathbb{Q}_p) \rightarrow \tilde{E}(\mathbb{F}_p)$, and by $\tilde{E}_{ns}(\mathbb{F}_p)$ the set of nonsingular points in $\tilde{E}(\mathbb{F}_p)$ (which form a group). Let  $E_0(\mathbb{Q}_p)=\{P\in E(\mathbb{Q}_p): \tilde{P} \in \tilde{E}_{ns}(\mathbb{F}_p) \}$ and $E_1(\mathbb{Q}_p)=\{P \in E(\mathbb{Q}_p): \tilde{P} = \tilde{\mathcal{O}}\}$. We know (see e.g. \cite[Chapter 7]{Sil}) that $E_1(\mathbb{Q}_p) \simeq \widehat{E}(p\mathbb{Z}_p)$, where $\widehat{E}/\mathbb{Z}_p$ is the formal group associated to $E$. More precisely, this isomorphism is given by
\begin{align*}
\widehat{E}(p\mathbb{Z}_p) &\rightarrow E_1(\mathbb{Q}_p)\\
z &\mapsto \left(\frac{z}{w(z)}, -\frac{1}{w(z)}\right),
\end{align*}
where $w(z)=z^3(1+\ldots)\in \mathbb{Z}_p[[z]]$ is the formal power series that (formally) relates functions $z=-\frac{x}{y}$ and $y=-\frac{1}{y}$ on $E$ (see \cite[Chapter 4]{Sil}). Moreover, the multiplication by $[p]$ map on $\widehat{E}(\mathbb{Z}_p)$ is an isomorphism, and satisfies the following identity (since $a_1=0$):
$$
[p]z=pz+O(z^3).
$$
Note that $\tilde{R}$ is a singular point, hence $R\not \in E_0(\mathbb{Q}_p)$, but in general we have that $E(\mathbb{Q}_p)/E_0(\mathbb{Q}_p)$ is a finite group. Next we will show that $R$ has order two in this group.

\begin{lemma} \label{lem:234}
The following is valid:
\begin{enumerate}
	\item [a)] $v_p(x([2]R))=0,$
	\item [b)] $v_p(x([3]R))=4,$
	\item [c)] $v_p(x([4]R))=-2$ and $v_p(y([4]R))=-3$.
\end{enumerate}
In particular, $[2]R \in E_0(\mathbb{Q}_p)$ and $[4]R\in E_1(\mathbb{Q}_p)$.
\end{lemma}

\begin{proof}
Explicit calculation shows
\begin{align*}
x([2]R) &=-\frac{3}{4}(t^2-6t+1)(t^2+6t+1)\\
x([3]R) &=-\frac{8}{9}(t^2+1)^4(t^2-18t+1)(t^2+18t+1)(t^2-6t+1)^{-2}(t^2+6t+1)^{-2}\\
x([4]R) &=\frac{15}{16}(t^2+1)^{-2}(t^2-18t+1)^{-2}(t^2+18t+1)^{-2}(t^2-6t+1)(t^2+6t+1)\\
&\mbox{}\times (t^4 - 36t^3 - 106t^2 - 36t + 1)(t^4 + \frac{118}{5}t^2 + 1)(t^4 + 36t^3 - 106t^2 + 36t + 1).
\end{align*}
Parts a) and b) follow immediately. For c) note that
$$(t^2+1)^2x([4]R) \equiv \frac{3}{16}(-18\cdot 18)^{-2}(-6\cdot 6)(108)(-108)(108)\not \equiv 0 \pmod {p}.$$
The second claim in c) follows from the cubic equation $y^2=x^3+\ldots$.
\end{proof}

\begin{lemma}
Let $n,m \in \mathbb{N}$. If $p\nmid m$ then
$$v_p(x([m\cdot p^n]([4]R)))=-2n-2.$$
\end{lemma}

\begin{proof}
Since $[4]R \in E_1(\mathbb{Q}_p)$, then $z=-\frac{x(4R)}{y(4R)} \in \widehat{E}(p\mathbb{Z}_p)$ with $v_p(z)=1$. It follows from the multiplication by $[k]$ formula that if $z\in p\mathbb{Z}_p$, then $v_p([m]z)=v_p(z)$ and $v_p([p]z)=v_p(z)+1$, hence $v_p([p^n][m]z)=n+1$. Since $x([p^n][m](4R))=\frac{[m\cdot p^n]z}{w([m\cdot p^n]z)}=\frac{1}{([m\cdot p^n]z)^2(1+\cdots)}$, the claim follows.
\end{proof}

\begin{lemma}
\label{lemma:valuations}
For $m \in \mathbb{N}$ the following applies:
\begin{enumerate}
	\item [a)] $v_p(x(R+[m]([4]R)))=4+v_p(m)$
	\item [b)] $v_p(x([2]R+[m]([4]R)))=0,$
	\item [c)] $v_p(x([3]R+[m]([4]R)))=4+v_p(m+1)$.
	\item [d)] $v_3(x([m]([3]R))) < 0$,
	\item [e)] $v_3(x(R+[m]([3]R))) > 0,$
	\item [f)]$v_3(x([2]R+[m]([3]R))) > 0.$
\end{enumerate}
\end{lemma}

\begin{proof}
Let $[x_1,y_1]$ and $[x_2, y_2]$ be two different points on $E$, and set %du
$[x_3,y_3]=[x_1,y_1]+[x_2,y_2]$. Then
$$x_3 = \lambda^2-a_2-x_1-x_2,$$
where $\lambda=\frac{y_2-y_1}{x_2-x_1}$, and $a_2= 3(t^2-3t+1)(t^2+3t+1)$.
\\
a) If $[x_1, y_1]=R=[0, (t^2+1)^3]$ and $[x_2,y_2]=[m]([4]R)$, then
\begin{align*}
x_3&=\lambda^2-a_2-x_2=\frac{(y_2-(t^2+1)^3)^2-a_2x_2^2-x_2^3}{x_2^2}\\
   &= \frac{(t^2+1)^4x_2+2(t^2+1)^6-2y_2(t^2+1)^3}{x_2^2}.
\end{align*}
Since $v_p(x_2)=-2(v_p(m)+1)$, it follows $v_p(\frac{(t^2+1)^4}{x_2})=6+2v_p(m)$, $v_p(\frac{2(t^2+1)^6}{x_2^2})=10+4v_p(m)$ and $v_p(\frac{-2y_2(t^2+1)^3}{x_2^2})=7+4v_p(m)-3(v_p(m)+1)=4+v_p(m)$, hence $v_p(x_3)=4+v_p(m)$.
\\
b) Since $[4]R$ reduces to $\tilde{\mathcal{O}} \in \tilde{E}_{ns}(\mathbb{F}_p)$, it follows that $\widetilde{([2]R+[4]R)}=\widetilde{[2]R}$ (note that $[2]R\in E_0(\mathbb{Q}_p)$), the claim follows from Lemma \ref{lem:234}.
\\
c) Since $[3]R+[m][4]R=-R+[m+1][4]R$ and $-R=[0,-(t^2+1)^3]$, the claim follows from the same argument as in a).
\\
d), e) and f) Notice that $\tilde{R}$ is a non-singular point (with $x(\tilde{R})=0$) in the order $3$ cyclic group $\tilde{E}_{ns}(\mathbb{F}_3)$ (the other two points are $\tilde{\mathcal{O}}$ and $(-\tilde{R})$). It follows $\widetilde{[m][3]R}=\tilde{\mathcal{O}}$,  $\widetilde{R+[m][3]R}=\tilde{R}$ and  $\widetilde{[2]R+[m][3]R}=-\tilde{R}$, hence $v_3(x([m][3]R)<0$, $v_3(x(R+[m][3]R)) > 0$, and $v_3(x([2]R+[m][3]R)) > 0$.
\end{proof}

\begin{lemma}
\label{lem:t}
Let $q|t$ be a prime and $m$ a positive integer. Then $x([m]R)\not \equiv -1 \pmod{q}$.
\end{lemma}

\begin{proof}
Note that mod $q$ reduction of point $R$ is a non-singular point, i.e. $R\in E_0(\mathbb{Q}_q)$. Hence $[m]R\in E_0(\mathbb{Q}_q)$ also has a non-singular reduction mod $q$. Since the only singular point in $\tilde{E}(\mathbb{F}_q)$ is the point $[-1,0]$, we conclude that $x([m]R)\not \equiv -1 \pmod{q}$.
\end{proof}

\begin{corollary}
\label{cor:integers}
Let  $t \ne 1$ be a positive integer such that the number $t^2+1$ is squarefree.
The elliptic curve $E''$ that corresponds to any multiple  $[m]R$, where $m>1$, has full rational $2$-torsion (i.e. the corresponding $a,b$ and $c$ are rational).
\end{corollary}

\begin{proof}
Let $p$ be an odd prime divisor of $t^2+1$. Then $p||t^2+1$. Lemma \ref{lemma:valuations} together with the Proposition \ref{prop:cases} implies that $E''$ has good or multiplicative reduction at $p$. If $p|t$ then by Lemma \ref{lem:t} we have $x([m]R)\not \equiv -1 \pmod{p}$, and part a) of Proposition \ref{prop:cases} implies that $E''$ has multiplicative reduction at $p$. Since Lemma \ref{lemma:valuations} implies $v_3(y([m]R))\le 0$  (note that $3\nmid t^2+1$), it follows from Corollary \ref{cor:conclusion} that $E''$ has rational $2$-torsion.
\end{proof}

To prove a rationality of $2$-torsion of $E''$ that corresponds to the multiple of $R$ for arbitrary rational $t$, we will need an effective version of Hilbert's Irreducibility Theorem. The following theorem was proved by D\"orge \cite{Do}.

\begin{theorem}
\label{thm:Hilbert}
If $f(X,t)$ is an irreducible polynomial with integral coefficients and if $R(N)$ is the number of integers $\tau$ such that $|\tau|<N$ and $f(x,\tau)$ is reducible, then $R(N)\le C N^{1-\alpha}$ where $\alpha$ and $C$ are certain positive constants.
\end{theorem}

We can now prove the main theorem of this section.

\begin{theorem}
\label{thm:all}
Let $t \notin \{-1,0,1\}$ be a rational number. Then the elliptic curve $E''$ that corresponds to any multiple  $[m]R$, for $m>1$, has full rational $2$-torsion.
\end{theorem}

\begin{proof}
Let
$$
g(X,t)=X^3+\frac{((t^2+1)^2x^{-1}+1)^2}{4}X^2+\frac{t^2((t^2+1)^2x^{-2}+x^{-1})}{2} X + \frac{t^4x^{-2}}{4}\in \mathbb{Q}(t)[X],
$$
where $x= x([m]R)\in \mathbb{Q}(t)$ (i.e. $E''$ is given by the equation $y^2=g(X,t)$).

Let $f(X,t)\in \mathbb{Z}[X,t]$ be a polynomial that one gets by multiplying out denominators in $g(X,t)$. It is enough to prove that that $f(X,t)$ is reducible. Suppose that $f(X,t)$ is irreducible. Then Theorem \ref{thm:Hilbert} implies that $\lim_N \frac{R(N)}{N}=0$. On the other hand, Corollary \ref{cor:integers} implies that $f(X,t)$, for $t\in \mathbb{Z}$, is reducible whenever $t^2+1$ is squarefree. It is well known that $t^2+1$ is squarefree for a positive proportion of integers $t$ (see \cite{E}), hence $\lim_N \frac{R(N)}{N} \ne 0$, and $f(X,t)$ is reducible. It is easy to see that $f(X,t)$ splits as a product of three linear polynomials in $X$, which implies that $E''$ has full rational $2$-torsion for every $t\in \mathbb{Q}$.
\end{proof}

\begin{remark}
Note that in order to rule out a ramification at $3$ of the field extension $\mathbb{Q}(E''[2])/\mathbb{Q}$ in Corollary \ref{cor:conclusion} we needed to know that $v_3( y([m]R)) \le 0$. This was needed for the proof of Corollary \ref{cor:conclusion} since there is $\mathbb{Z}/3\mathbb{Z}$ extension $\mathbb{Q}(\alpha)$ unramified outside $3$, where $\alpha^3-3\alpha+1=0$. Also, for example, for $t=17$ and point $P=[35000, 40986000]$ (note that $3|40986000$), one gets an elliptic curve $E''$ with additive reduction at $3$.

One could avoid mod $3$ analysis altogether and obtain the Theorem \ref{thm:all} by restricting in Corollary \ref{cor:integers} on $t$'s divisible by $3$ (then multiplicative reduction at $3$ immediately follows) and by using in the proof of Theorem \ref{thm:all} the fact that $9t^2+1$ is squarefree for positive proportion of $t$'s.
\end{remark}

\begin{remark}
Not all triples $\{a,b,c\}$ obtained using this construction are rational.
For example for $t=31$ (when the rank of $E$ is $2$) %du
and point $(x,y)=[-150072,682327360]$ (which is not a multiple of $R$) %du
the curve $E''$ has no rational $2$-torsion. It has additive reduction at $13, 31$ and $37$. Note that $31^2+1=2\cdot 13 \cdot 37$, and $v_{13}(x)=2$, $v_{37}(x)=1$ and $x \equiv -1 \pmod{31}$.
\end{remark}

%Concluding section
\section{Explicit formulas}

In the previous section we proved indirectly (using Hilbert's Irreducibility Therem) that Diophantine triples $\{a,b,c \}$ associated to the multiples $mR$, for $m>1$, of point $R=[0,(t^2+1)^3] \in E$ are rational. In this section, following referee's suggestion, we provide explicit formulas for these triples.

The two torsion subgroup $E''[2]$ (see \eqref{eq:E''}) defines the plane curve $C$ over $\mathbb{Q}(t)$ given by the equation
\[
C: \qquad X^3+\frac{((t^2+1)^2u+1)^2}{4}X^2+\frac{t^2((t^2+1)^2u^2+u)}{2} X + \frac{t^4u^2}{4}=0.
\]

By resolving singularity at the point $(X,u)=(0,0)$ we obtain a rational parametrization of this genus $0$ curve

\[
X(w)=-\frac{w^2}{4(t^2+1)^2},\qquad u(w)=\frac{w-t^2-1}{(t^2+1)(-w^2/4+t^2)}X, \quad \textrm{where}\quad w\in\mathbb{Q}(t).
\]

From the construction of $E''$, we know that $u^{-1}(w)$ defines the $x$-coordinate of a $\mathbb{Q}(t)$-rational point on $E$. By substituting $u^{-1}(w)$ in \eqref{eq:E}, we obtain a curve which is birationally equivalent to the curve
\[
y^2=(w-(t^2+1))(w^3-3t^2w-t^2(t^2+1)),
\]
which again is birationally equivalent to the elliptic curve $E_{*}$
\begin{equation}\label{eq:E_}
E_{*}: y^2 = x^3 + 3(t^2-3t+1)(t^2+3t+1)x^2+3(t^2+1)^2(t^4-178t^2+1)x+(t^2+1)^2(t^4+110t^2+1)^2.
\end{equation}
The curve $E_{*}$ is $3$-isogenous to the curve $E$ (the kernel of the isogeny $\phi:E_{*} \rightarrow E$ is the subgroup of $E_{*}(\mathbb{Q}(t))$ of order $3$ generated by the point $T=[-(t^2-6t+1)(t^2+6t+1),27t(t-1)^2(t+1)^2 ]$). Moreover, we have that $\phi(P)=R$, where $P=[-(t^2+1)(t^2+18t+1), 27t(t+1)^2(t^2+1)]\in E_{*}(\mathbb{Q}(t))$. We can express $X$ and $u$ using coordinates $(x,y)$ on $E_{*}$. We have that $X=X(x,y)=X(w)$ and $u=u(w)$ where
$$w=w(x,y)=\frac{27y+9rx+s}{2(9x+v)},$$ with $(v,r,s)=\left(\frac{5}{4} t^4 + \frac{59}{2}t^2 + \frac{5}{4}, -\frac{3}{2}( t^2 + 1), -\frac{27}{8}(t-1)^2(t+1)^2(t^2+1)\right).$ One can check that for any point $Q\in E_{*}$, we have that
$u(w(Q))^{-1}$ is equal to the $x$-coordinate of the point $\phi(Q)+R\in E$, where $w(Q)=w\left(x(Q),y(Q)\right)$.

In particular, the triple $\{a, b, c\}$ associated to the multiple $mR$, with $m>1$, can be written in the following way. Let $X_1=-\left(\frac{t^2+1}{t}\right)^2 X((m-1)P)-1$, $X_2=-\left(\frac{t^2+1}{t}\right)^2 X((m-1)P+T)-1$ and $X_3=-\left(\frac{t^2+1}{t}\right)^2X((m-1)P+2T)-1$, i.e. $X_1=ab$, $X_2=ac$ and $X_3=bc$. Then
\[
\{a, b, c \}=\left\{ \sqrt{\frac{X_1 X_2}{X_3}}, \sqrt{\frac{X_1 X_3}{X_2}}, \sqrt{\frac{X_2 X_3}{X_1}}\right\}.
\]

\section{Some concluding remarks}  \label{sec6}

\begin{remark}
\label{rmz2z6}
Note that for rational Diophantine triples $\{a,b,c\}$ satisfying condition (\ref{mi1}), the induced elliptic curve has torsion group  $\mathbb{Z}/2\mathbb{Z} \times \mathbb{Z}/6\mathbb{Z}$, since it contains the point $S$ of order three. It is an open problem whether this torsion group is possible for elliptic curves induced by an integer Diophantine triple (see e.g. \cite{D-M,Mi}). On the other hand, examples of elliptic curves, induced by rational Diophantine triples, with torsion group $\mathbb{Z}/2\mathbb{Z} \times \mathbb{Z}/6\mathbb{Z}$ and rank equal to $1$, $2$, $3$ and $4$ can be found in \cite{D-mw} (for examples of elliptic curves with torsion groups $\mathbb{Z}/2\mathbb{Z} \times \mathbb{Z}/4\mathbb{Z}$ and $\mathbb{Z}/2\mathbb{Z} \times \mathbb{Z}/8\mathbb{Z}$ and high rank see \cite{D-P,D-mw}). These examples were obtained from the condition $3P=\mathcal{O}$. The first example of a rational Diophantine triple with positive elements satisfying the condition $3S=\mathcal{O}$ was given in \cite{M-thesis}. This triple was
\begin{equation}
\label{mi-triple}
\left\{ \frac{36534805866201747}{2323780774755404},
\frac{1065197767305747}{13609226201091404},
\frac{3802080647508196}{6238332600753747} \right\},
\end{equation}
and it can be obtained by taking $\sigma_3=3/4$ in (\ref{mi3}). As before, we can transform the quartic to the elliptic curve $Y^2 = X^3+1512X+33588$. This curve has rank $1$ with the generator $U=[-11, 125]$, and the point $6U$ is the smallest multiple of $U$ which produces the triple with all positive elements, which is exactly the triple (\ref{mi-triple}). By applying the construction described above, we can extend this triple to infinitely many rational Diophantine sextuples. One such extension is by the following three elements
\begin{eqnarray*}
\frac{143947705777192337861060209232361164451}{159554724645105598216911731751641945996}, \\[8pt]
\frac{27566706033755538837165550223247346480484}{28811406145997336392588207503703089363}, \\[8pt]
\frac{5959833363761715860447368794188813530156}{3132578990197106752312648160330628526617}.
\end{eqnarray*}
\end{remark}

\begin{remark} \label{rmrank}
Our parametric formula for the rational Diophantine sextuples $\{a,b,c,d,e,f\}$ from Section \ref{sec3} can be used to obtain an elliptic curve over $\mathbb{Q}(t)$ with reasonably high rank. Indeed, consider the curve
$$
C: \quad y^2= (dx+1)(ex+1)(fx+1).
$$
It has three obvious points of order two, but also points with $x$-coordinates
$$
0, \quad \frac{1}{def},\quad a, \quad b, \quad c.
$$
We claim that these five points are independent points of infinite order on the curve $C$ over $\mathbb{Q}(t)$. Since the specialization map is a homomorphism, it suffices to find one parameter $t$ for which these five points become independent points of infinite order on the specialized curve $C_t$ over $\mathbb{Q}$. It is easy to check that $t=2$ is one such specialization (by checking that the discriminant of the corresponding height matrix is nonzero). Therefore, we proved that the rank of $C$ over $\mathbb{Q}(t)$ is $\geq 5$, which ties the published record from \cite{A-D-P} for the generic rank of elliptic curves over $\mathbb{Q}(t)$ induced by Diophantine triples.
\end{remark}

\begin{remark}
The question whether there exist any rational Diophantne septuple remains open.
It seems unlikely that a construction based on the application of Proposition \ref{matija} could be used to obtain a rational Diophantine septuple. The point $S'=[1,rst]$ must be of order $3$ (see the discussion after Proposition \ref{matija}), which makes it unsuitable for extending the given triplet to the septuple (at least with our approach), since the natural choice for extending the set, the number $x(T+2S')$, is already in the set (it is equal to $x(T-S')$).

\end{remark}

\section*{Acknowledgements}
We would like to thank Filip Najman for his comments on the first draft of the paper, and to the referee for suggesting us to describe the two torsion of the elliptic curve $E''[2]$ explicitly over the function field of the elliptic curve that is 3-isogenous to the elliptic curve $E$. A.D. was supported by the Croatian Science Foundation under the project no. 6422.

%References

\end{document}